\documentclass{amsart}

\newcommand{\bi}{{\bf i}}
\newcommand{\bj}{{\bf j}}
\newcommand{\bk}{{\bf k}}

\usepackage{graphicx, wrapfig}
\usepackage{amssymb,amscd,verbatim, amsthm, amsmath}

\newcommand{\Z}{{\mathbb Z}}
\newcommand{\N}{{\mathbb N}}

\newcommand{\qab}{Q_{a,b}}
\newcommand{\gab}{g_{a,b}}
\newcommand{\qtab}{Q^2_{a,b}}
\newcommand{\al}{\alpha}

\newtheorem{theorem}{Theorem}[section]
\newtheorem{definition}[theorem]{Definition}

\newtheorem{conjecture}[theorem]{Conjecture}
\newtheorem{lemma}[theorem]{Lemma}

\begin{document}

\title{Sums of squares in Quaternion rings}
\author[Cooke]{Anna Cooke}
\author[Hamblen]{Spencer Hamblen}
\author[Whitfield]{Sam Whitfield}
\thanks{Research supported by the McDaniel College Student-Faculty Summer Research Fund}

\begin{abstract}
Lagrange's Four Squares Theorem states that any positive integer can be expressed as the sum of four integer squares.  We investigate the analogous question over Quaternion rings, focusing on squares of elements of Quaternion rings with integer coefficients.  We determine the minimum necessary number of squares for infinitely many Quaternion rings, and give global upper and lower bounds.
\end{abstract}

\maketitle

\section{Introduction and Definitions}

\subsection*{Waring's Problem}

\begin{theorem}[Waring's Problem/Hilbert-Waring Theorem]
For every integer $k \geq 2$ there exists a positive integer $g(k)$ such that every positive integer is the sum of at most $g(k)$ $k$-th powers of integers.
\end{theorem}

Generalizations of Waring's Problem have been studied in a variety of settings (for example, number fields \cite{siegel} and polynomial rings over finite fields \cite{car}).  Additionally, calculation of the exact values of $g(k)$ for all $k \geq 2$ was completed only relatively recently.  For an excellent and thorough exposition of the research on Waring's Problem and its generalizations, see Vaughan and Wooley \cite{wooley}.   We will examine a generalization of Waring's Problem to Quaternion rings. 

\begin{definition} Let $\qab$ denote the Quaternion ring
\[ \{\alpha_0 + \alpha_1 \bi + \alpha_2 \bj + \alpha_3 \bk \mid \alpha_n,a,b \in \Z, \bi^2 = -a, \bj^2 = -b, \bi\bj=-\bj\bi=\bk\}.\]
Let $\qab^n$ denote the additive group generated by all $n$th powers in $\qab$.
\end{definition}

Note here that $\bk^2 = -ab$, and that if $a = b = 1$, we have what are called the {\em Lipschitz Quaternions}.  We then have the following analogue of Waring's Problem.

\begin{conjecture}
For every integer $k \geq 2$ and all positive integers $a,b$ there exists a positive integer $\gab(k)$ such that every element of $\qab^k$ can be written as the sum of at most $\gab(k)$ $k$-th powers of elements of $\qab$.
\end{conjecture}

\subsection*{Main Results}

We will examine sums of squares in Quaternion rings; that is, when $k=2$.   We are therefore looking to generalize Lagrange's Four Squares Theorem, the inspiration for Waring's initial conjecture.

\begin{theorem}[Lagrange's Four Squares Theorem] \label{lag4}
Any positive integer can be written as the sum of four integer squares.
\end{theorem}

We prove the following general result giving the upper and lower bounds for $\gab(2)$ for any positive integers $a$ and $b$.

\begin{theorem} \label{sqbounds}
For all positive integers $a,b$, we have 
\[3 \leq \gab(2) \leq 5.\]
Additionally, each possible value of $\gab (2)$ (i.e., 3, 4, and 5) occurs infinitely often.
\end{theorem}

We prove the general upper and lower bounds in Section \ref{sec:gbc2}; more specific results, including the proof of the latter half of Theorem \ref{sqbounds}, are given in Section \ref{valueg}.  Note that for any positive integers $a$ and $b$, $\qab$ and $Q_{b,a}$ are naturally isomorphic; we therefore generally assume that $a \leq b$.

\section{Squares of Quaternions -- Upper and Lower Bounds \label{sec:gbc2}}

In this section we prove the upper and lower bounds of Theorem \ref{sqbounds}.  We will use the following classical result on sums of squares extensively; for this result and a more general look at sums of squares of integers see \cite{sumsofsq}.

\begin{theorem}[Legendre's Three Squares Theorem]
\label{leg3}
A positive integer $N$ can be written as the sum of three integer squares if and only if $N$ is not of the form $4^m(8\ell + 7)$ with $\ell, m$ non-negative integers.
\end{theorem}

To study $\gab (2)$, we first need to establish the general form of squares of quaternions, and to characterize elements of $\qtab$.

	Let $\alpha=\alpha_0+\alpha_1\bi+\alpha_2\bj+\alpha_3\bk \in \qab$.  We call $\alpha_0$ the {\em real} part of $\alpha$ and $\alpha_1\bi+\alpha_2\bj+\alpha_3\bk$ the {\em pure} part of $\alpha$, with $\alpha_1, \alpha_2, \alpha_3$ the {\em pure coefficients}.  Then note that 
\begin{equation} \label{al2}
	\alpha^2=\alpha_0^2-a\alpha_1^2-b\alpha_2^2 - ab\alpha_3^2 + 2\alpha_0\alpha_1\bi + 2\alpha_0\alpha_2\bj + 2\alpha_0\alpha_3\bk.
\end{equation}    
We therefore have that all the pure coefficients of squares of quaternions, and therefore the pure coefficients of all elements of $\qab^2$, are even.  Additionally, any set of even pure coefficients can be achieved (for example, set $\alpha_0 = 1$ in Equation (\ref{al2})), as can any negative real coefficient (since we are assuming $a,b \geq 1$).  We therefore have
\begin{equation} \label{sqform}
	\qab^2 = \{\alpha_0 + 2\alpha_1 \bi + 2\alpha_2 \bj + 2\alpha_3 \bk \mid \alpha_n \in \mathbb{Z}\}.
\end{equation}    
\medskip

In 1946, Niven computed $g_{1,1}(2)$ and studied extensions of Waring's Problem in other various settings, including the complex numbers.

\begin{theorem}[Niven \cite{nivenquat}] 
Every element in $Q^2_{1,1}$ can be written as the sum of at most three squares in $Q_{1,1}$.  Additionally, $6+2\bi$ is not expressible as the sum of two squares in $Q_{1,1}$, so $g_{1,1}(2) = 3$.
\end{theorem}

We extend this result to $\qab$ for all positive integers $a,b$.  The proofs for the lower bounds are similar to Niven's work (i.e., finding examples); the proofs for the upper bounds take more work.

\begin{lemma} \label{lbsq}
Suppose $a$ and $b$ are positive integers.  Then if
\begin{itemize}
	\item $a \equiv 1$ or $2 \bmod 4$, then $2 + 2\bi$ is not expressible as the sum of two squares in $\qab$; and
	\item $a \equiv 0$ or $3 \bmod 4$, then $4 + 2\bi$ is not expressible as the sum of two squares in $\qab$.
\end{itemize}    
\end{lemma}

\begin{proof}
Let $x = x_0 + x_1 \bi + x_2 \bj + x_3 \bk$, and $y =  y_0 + y_1 \bi + y_2 \bj + y_3 \bk$, with $x_m, y_n \in \Z$ for $m,n \in \{0,1,2,3\}$. 
Then if $x^2 + y^2 = \alpha$ with $\alpha = \alpha_0 + 2\alpha_1 \bi + 2\alpha_2 \bj + 2\alpha_3 \bk \in \qtab$, we have
\begin{align}
	\alpha_0 &= x_0^2 + y_0^2 - a(x_1^2 + y_1^2) - b(x_2^2 + y_2^2) - ab (x_3^2 + y_3^2) \label{lb1}\\
	\alpha_1 &= x_0x_1 + y_0y_1 \label{lb2}\\
	\alpha_2 &= x_0x_2 + y_0y_2 \label{lb3}\\
	\alpha_3 &= x_0x_3 + y_0y_3. \label{lb4}
\end{align}

\underline{Case 1: ($a \equiv 1,2 \bmod 4$)}  Suppose $a \equiv 1,2 \bmod 4$, and let $\alpha = 2 + 2\bi$, so that $\alpha_0 = 2$, $\alpha_1 = 1$, and $\alpha_2=\alpha_3=0$.  Since $\alpha_1 = 1$, Equation (\ref{lb2}) and Bezout's Identity then imply that $x_0$ and $y_0$ must be relatively prime, since they have a linear combination equal to 1.  Then, by Equation (\ref{lb3}), we must have $x_0 | y_2$ and $y_0 | x_2$.  However, since $b \geq 1$, if $x_2, y_2 \neq 0$, Equation (\ref{lb1}) then implies that $\alpha_0 \leq 0$.  As $\alpha_0 = 2$, we must have $x_2 = y_2 = 0$.  A similar argument using Equation (\ref{lb4}) implies that $x_3 = y_3 = 0$.

By Equation (\ref{lb2}), since $\alpha_1 = 1$, we have that exactly one of the products $x_0x_1$ and $y_0y_1$ must be odd; we therefore assume $y_0$ and $y_1$ are odd.
The following table then shows that Equation (\ref{lb1}) has no solutions mod 4 if $a \equiv 1,2 \bmod 4$: 

\begin{center}
$\begin{array}{c|c|c}
	x_0 & x_1 & \text{Equation (\ref{lb1})} \bmod 4 \\ \hline
	\text{even} & \text{odd} & \alpha_0 = 2 \equiv 1 - 2a \\
	\text{even} & \text{even} & \alpha_0 = 2 \equiv 1 - a \\
	\text{odd} & \text{even} & \alpha_0 = 2 \equiv 2 - a
\end{array}$
\end{center}

Therefore $2 + 2\bi$ cannot be written as the sum of two squares in $\qab$.  
\medskip

\underline{Case 2: ($a \equiv 0,3 \bmod 4$)}  Suppose $a \equiv 0,3 \bmod 4$.  Then let $\alpha = 4 + 2\bi$.  By the same argument as above, we get 3 possibilities for Equation (\ref{lb1}) mod 4, none of which have solutions.  Therefore $4 + 2\bi$ cannot be written as the sum of two squares in $\qab$. 
\end{proof}

As both $2+2\bi$ and $4+2\bi$ are in $\qab^2$, this gives us the lower bound in Theorem \ref{sqbounds}.  We then turn to the upper bound; we establish an algorithm for expressing every element as a sum of squares.

\begin{lemma} \label{ubsq}
Every element in $\qtab$ can be written as a sum of at most five squares in $\qab$.
\end{lemma}

\begin{proof}
Let $\alpha=\alpha_0+2\alpha_1\bi+2\alpha_2\bj+2\alpha_3\bk \in \qab^2$; We want to show that we can represent $\alpha$ as a sum of squares of no more than five quaternions. 

Let $v = 1 + U\bi + \alpha_2\bj + \alpha_3 \bk$ for some $U\in \Z$, and note that
\[\alpha - v^2 = \alpha_0 - 1 + aU^2 + b\alpha_2^2 + ab\alpha_3^2 + 2(\alpha_1 - U)\bi.\]
If we also let $A = \alpha_0 - 1 +a\alpha_1^2 + b\alpha_2^2 + ab\alpha_3^2$, we have
\begin{equation} \label{vau}
\alpha - v^2 = A + a(U^2 - \alpha_1^2) + 2(\alpha_1 - U)\bi.
\end{equation}
We then have three cases: (1) when $A \geq 0$, (2) when $A < 0$ and $A$ cannot be written as $4^m (8 \ell + 7)$ for any non-negative integer $m$ and $\ell \in \Z$, and (3) when $A < 0$ and $A = 4^m (8 \ell + 7)$ for some non-negative integer $m$ and $\ell \in \Z$.

\underline{Case 1: $A \geq 0$.}  If $A \geq 0$, then by Lagrange's Four Squares Theorem (Theorem \ref{lag4}), there exists $w,x,y,z \in \Z$ such that $A = w^2 + x^2 + y^2 + z^2$.  Letting $U = \alpha_1$,  Equation (\ref{vau}) becomes
\[\alpha - v^2 = A = w^2 + x^2 + y^2 + z^2,\]
so we can represent $\alpha$ as the sum of five squares.

\underline{Case 2: $A < 0$ and $A \neq 4^m (8 \ell + 7)$.} In this case we again let $U = \alpha_1$, so that $\alpha - v^2 = A$.  Then let $e_1$ be the greatest exponent of 4 such that $4^{e_1}$ divides $A$, and let $e_2$ be the least exponent of 4 such that $4^{2e_2} + A \geq 0$.  We then let $e = \max\{e_1+1,e_2\}$, and let $w = 4^e \bi$.

We then have $\alpha - v^2 - w^2 = A + a4^{2e} \geq 0$.  Additionally, since $2e \geq 2e_1 + 2$, if $A$ cannot be written in the form $4^m (8 \ell + 7)$, then neither can $A + 4^{2e}$.  Therefore by Legendre's Three Squares Theorem (Theorem \ref{leg3}), there exist $x,y,z \in \Z$ such that $A+4^{2e} = x^2 + y^2 + z^2$.  So
\[\alpha - v^2 - w^2 = A + 4^{2e} = x^2 + y^2 + z^2,\]
so we can represent $\alpha$ as the sum of five squares.

\underline{Case 3: $A < 0$ and $A = 4^m (8 \ell + 7)$.}  We first treat the case when $m > 0$.  Here we let
\[w = 2^{m-1} + \left(\frac{\alpha_1 - U}{2^{m-1}}\right)\bi\]
and choose $U = \alpha_1 + 2^{m-1}U_1$, where $U_1$ satisfies the following 3 conditions:
\begin{description}
 	\item[(a)] $4^{m+1} | U_1$, 
 	\item[(b)] $U_1 > -\displaystyle\frac{2^{m}\alpha_1}{4^{m-1}+1}$, and
 	\item[(c)] $U_1 > \displaystyle\frac{A-4^{m-1}}{a}.$
\end{description}
Note that it is always possible to meet these conditions; for example, $U_1 = 4^{m+1}|A| \cdot \max\{1,|\alpha_1|\}$ satisfies all three.  We then have
\begin{align*}
	\alpha - v^2 - w^2 & = \left(A + a(U^2 - \alpha_1^2) + 2(\alpha_1 - U)\bi\right) - \left(4^{m-1} + 2 (\alpha_1 - U)\bi - a \left(\frac{\alpha_1 - U}{2^{m-1}}\right)^2\right) \\
    & = A + a(\alpha_1^2 + 2^m \alpha_1U_1+ 4^{m-1} U_1^2 - \alpha_1^2) - 4^{m-1} + aU_1^2 \\
    & = A - 4^{m-1} + aU_1\left(2^m \alpha_1 + (4^{m-1}+1)U_1\right).
\end{align*}
Note that condition (b) on $U_1$ ensures the quantity in parentheses must be positive, and condition (c) ensures that $\alpha - v^2 - w^2$ is positive.  Letting $A = 4^m(8 \ell + 7)$ and (since $4^{m+1} | U_1$) the remainder of the equation equals $4^{m+1}\ell_1$ for some $\ell_1 \in \Z$, we have
\begin{align*}
	\alpha - v^2 - w^2 & = 4^m(8 \ell + 7) - 4^{m-1} + 4^{m+1}\ell_1 \\
    & = 4^{m-1}\left[4(8 \ell + 7) - 1 + 16 \ell_1\right] \\
    & = 4^{m-1}\left[8(4\ell + 3 + 2\ell_1) + 3\right],
\end{align*}
Since this is not of the form excluded by Legendre's Three Squares Theorem, there exist $x,y,z \in \Z$ such that $\alpha - v^2 - w^2 = x^2 + y^2 + z^2$, so we can represent $\alpha$ as the sum of five squares.

Lastly, we treat the case when $A = 8 \ell + 7$ for some negative integer $\ell$.  Here we let $U = \alpha_1 + U_1$ and $w = 1 + U_1 \bi$, choosing $U_1$ such that $8 \mid U_1$ and $U_1 > \max\{|A|, |\alpha_1|\}$.  Then
\begin{align*}
	\alpha - v^2 - w^2 & = \left(A + a(U^2 - \alpha_1^2) + 2(\alpha_1 - U)\bi\right) - (1 - U_1\bi)^2 \\
    & = A + a(2\alpha_1U_1+ U_1^2) - 1 + aU_1^2 \\
    & = 8 \ell +6 + 8 \ell_1,
\end{align*}
where we have $\ell_1  + \ell \geq 0$ by the conditions on $U_1$.  Since this is a positive number that is 6 mod 8, it is expressible as the sum of 3 integer squares by Legendre's Three Squares Theorem.  So we can represent $\alpha$ as the sum of five squares here and in all cases.
\end{proof}

Lemmas \ref{lbsq} and \ref{ubsq} combined give the bounds for $\gab (2)$ in Theorem \ref{sqbounds}.

\section{Values of $\gab(2)$ \label{valueg}}

In this section, we establish exact values for $\gab(2)$ for several infinite families of Quaternion rings, and for each of the possible values of $\gab(2)$.  We note that the methods for showing each are different: for example, to show $\gab(2)=3$, all we need is an algorithm to express every element in $\qtab$ as a sum of 3 squares, and to show $\gab(2)=5$, all we need is to find an element that cannot be expressed as the sum of 4 squares.

\subsection{$\gab(2) = 3$}

We examine $Q_{1,b}$, where $b \in \N$.  We can view $Q_{1,b}$ as an extension of the Gaussian integers $\Z[\sqrt{-1}]=\{x+y\sqrt{-1} \mid x,y \in \Z\}$ by adjoining $\bj$ and $\bk$.  The following Lemma then provides a shortcut for representing elements of $Q_{1,b}$ as sums of squares.
    
\begin{lemma}[Theorem 2 of \cite{nivengauss}] \label{nivgausslem}
    	The equation $\alpha_0+2\alpha_1\bi=x^2+y^2$ is solvable in $\mathbb{Z}[\sqrt{-1}]$ if $\alpha_0/2$ and $\alpha_1$ are not both odd integers.
\end{lemma}

Note that this Lemma also implies that $g_{\Z[\sqrt{-1}]}(2) = 3$.  

\begin{theorem} 
For all $b \in \N$, every element in $Q_{1,b}^2$ can be written as the sum of at most three squares in $Q_{1,b}$.  Therefore $g_{1,b} (2) = 3$ for all $b \in \N$.
\end{theorem}

\begin{proof}
	Let $\alpha=\alpha_0+2\alpha_1\bi+2\alpha_2\bj+2\alpha_3\bk \in Q^2_{1,b}$; we wish to find $x,y,z \in Q_{1,b}$ such that $\alpha = x^2 + y^2 + z^2$.  Since $\Z[\sqrt{-1}] \subset Q_{1,b}$, Lemma \ref{nivgausslem} implies that it is sufficient to find $z \in Q_{1,b}$ such that $\alpha - z^2 \in \Z[\sqrt{-1}]$ and satisfies the hypotheses of Lemma \ref{nivgausslem}. 
    
    Therefore, let $z=1+U\bi+\alpha_2\bj+\alpha_3\bk$, where $U=0$ if $\alpha_1$ is even and $U=1$ if $\alpha_1$ is odd. We then examine $\alpha-z^2$.
    \begin{align*}
    	\alpha-z^2&=\alpha_0+2\alpha_1\bi+2\alpha_2\bj+2\alpha_3\bk - 1+U^2+b\alpha_2^2+b\alpha_3^2 - 2U\bi -2\alpha_2\bj-2\alpha_3\bk\\
    	&=\alpha_0-1 + U^2+b\alpha_2^2+b\alpha_3^2+2(\alpha_1-U)\bi
    \end{align*}
Note that if $\alpha_1$ is even, then $U=0$, so $\alpha_1-U$ is even; conversely, if $\alpha_1$ is odd, then $U=1$, so $\alpha_1-U$ is again even.  We can therefore apply Lemma \ref{nivgausslem} to find $x,y \in \Z[\sqrt{-1}] \subset Q_{1,b}$ such that $\alpha - z^2 = x^2 + y^2$. 
\end{proof}

We note that the proof relies on the fact that squares in the Gaussian integers can be easily characterized.  This is not generally true of imaginary quadratic fields (see \cite{eljoseph} and Theorem 3 of \cite{nivengauss}).

\subsection{$\gab(2) = 4$}

We combine a standard lower bound proof and a constructive upper bound proof to find a family of Quaternion rings with $\gab(2) = 4$.

\begin{lemma} \label{43lb}
There exist elements in $Q_{4m,4n+3}^2$ that are not the sum of three squares.
\end{lemma}

\begin{proof}
	Suppose that there exist $x,y,z \in Q_{4m,4n+3}$ such that  $x^2 + y^2 + z^2 = 9 + 2\bj$.  Letting
	\begin{align*}
		x &=x_0+x_1\bi+x_2\bj+x_3\bk\\
		y &=y_0+y_1\bi+y_2\bj+y_3\bk\\
		z &=z_0+z_1\bi+z_2\bj+z_3\bk,
	\end{align*}
the resulting equations for the real and $\bj$ coefficients of $9+2\bj$ are, respectively:
	\begin{align}
		\begin{split}
			x_{0}^{2}+y_{0}^{2}+z_{0}^{2} - 4m(x_{1}^{2}+y_{1}^{2}+z_{1}^{2}) &- (4n+3)(x_{2}^{2}+y_{2}^{2}+z_{2}^{2})\\
			& - (4m)(4n+3)(x_{3}^{2}+y_{3}^{2}+z_{3}^{2})= 9
		\end{split}\label{43r}\\
		x_{0}x_{2}+y_{0}y_{2}+z_{0}z_{2}&=1.	\label{43j}
	\end{align}
	
	Examining Equation (\ref{43r}) mod 4, we have:
	\begin{equation}
	 x_{0}^{2}+y_{0}^{2}+z_{0}^{2} + x_{2}^{2}+y_{2}^{2}+z_{2}^{2} \equiv1\bmod 4. \label{43rm4}
	\end{equation}
	
	Recall then that for all integers $\ell$, we have $\ell^2\equiv 0 \bmod 4$ (if $\ell$ is even) or $\ell^2\equiv 1 \bmod 4$ (if $\ell$ is odd).  From this we have two possibilities that satisfy Equation (\ref{43rm4}): we must have either 1 or 5 of $x_{0},y_{0},z_{0},x_{2},y_{2},z_{2}$ odd in order for the left side of Equation (\ref{43rm4}) to sum to 1 mod 4. 
    
If only one of the terms is odd, then the left side of Equation (\ref{43j}) will be even since the lone odd term must be multiplied by an even term, and therefore cannot equal 1.  Likewise, if there are 5 odd terms, the left side of Equation (\ref{43j}) will be the sum of two odd terms and one even term, which cannot sum to 1.

Since Equations (\ref{43r}) and (\ref{43j}) cannot simultaneously be satisfied, $9 + 2 \bj$ cannot be expressed as the sum of three squares in $Q^{2}_{4m,4n+3}$.
\end{proof}

\subsubsection*{When $a$ is a Sum of 2 Integer Squares}

When $a$ is a sum of integer squares, we can construct an algorithm to express elements of $\qtab$ as the sum of 4 squares.  This gives us 
a general result when combined with the lower bound results of Lemma \ref{43lb}.

\begin{lemma} \label{toad}
	Every element of $\qtab$ is the sum of at most four squares in $\qab$ in the following two cases:
	\begin{itemize}
    	\item $a = n_1^2 + n_2^2$ with $\gcd(n_1,n_2)=1$; or
        \item $a = n_1^2 + n_2^2$ with $\gcd(n_1,n_2)=2$ and $n_1 \equiv 0 \bmod 4$, and $b \not\equiv 0 \bmod 4$. 
    \end{itemize}    
\end{lemma}

Note that we allow $n_1 = 0$ only if $n_2 = 1$ or 2; in the latter case we get $a=4$, which will be useful in light of Lemma \ref{43lb}.

\begin{proof}
Let $\alpha = \alpha_{0}+2\alpha_{1}\bi+2\alpha_{2}\bj+2\alpha_{3}\bk$.  If we let $z=1+\alpha_{1}\bi+\alpha_{2}\bj+\alpha_{3}\bk \in \qab$, then $\alpha-z^{2} \in \Z$.  We claim that every integer can be represented as the sum of three squares in $\qab$; we could then represent $\alpha$ as the sum of four squares.

Let $x = n_1 \ell + r$, $y = n_2 \ell + s$, and $w = \ell \bi + \delta \bj$, for some $\ell, r, s, \delta \in \Z$.  We then have
\begin{align} 
	x^2 + y^2 + w^2 & = (n_1 \ell + r)^2  + (n_2 \ell + s)^2 + (\ell \bi + \delta \bj)^2 \notag\\
    & = 2(r n_1 + s n_2)\ell + r^2 + s^2 - b\delta^2.
    \label{toadex}
\end{align}

Our method will be to choose $r$ and $s$ to determine a ``modulus'' ($r n_1 + s n_2$) and residue class ($r^2 + s^2 - b\delta^2$).  Since $\ell$ is independent of $r$ and $s$, we will therefore be able to represent every integer in that residue class.  (We will only use $\delta$ in one particularly troublesome case.)

Recall that by Bezout's Identity there exist $r_0, s_0 \in \Z$ such that $r_0n_1 + s_0n_2= \gcd(n_1,n_2) \in \{1,2\}$; these will inform our choices of $r$ and $s$.  We then have three cases (relabeling if necessary) that we address separately: 
\begin{enumerate}
	\item[(a)] $n_1$ odd, $n_2$ even, and $\gcd(n_1,n_2) = 1$;		
    \item[(b)] $n_1, n_2$ odd, and $\gcd(n_1,n_2) = 1$; and
 	\item[(c)] $n_1/2$ even, $n_2/2$ odd, and $\gcd(n_1,n_2) = 2$.
\end{enumerate}    
    
\underline{Case (a)}: Our modulus here will be 2.  Note that if $r=r_0$, $s=s_0$, and $\delta = 0$, we have from Equation (\ref{toadex})
\begin{equation*} \label{toad1a}
    x^2+y^2+w^2 = 2\ell + r_0^2 + s_0^2.
\end{equation*}
Next, if $r=r_0-n_2$, $s=s_0+n_1$, and $\delta=0$, Equation (\ref{toadex}) yields
\begin{equation*} \label{toad1b}
    x^2+y^2+w^2 = 2\ell + (r_0-n_2)^2 + (s_0+n_1)^2.
\end{equation*}

Recalling that $n_1$ is assumed to be odd and $n_2$ is assumed to be even, we necessarily have that $r_0^2 + s_0^2$ and $(r_0-n_2)^2 + (s_0+n_1)^2$ cover all residue classes mod 2 with the two equations above.  With a  proper choice of $\ell$, we can therefore directly find $x,y,w \in \qab$ such that $\alpha - z^2 = x^2+y^2+w^2$, and so we can write $\alpha$ as a sum of four squares in $\qab$.

\underline{Case (b)}: Our modulus here will be 4.  Since $n_1$ and $n_2$ are here both odd, we may assume that without loss of generality that $r_0$ is odd and $s_0$ is even.

We then use three choices of $r$ and $s$ to represent all possible residue classes mod 4; we let $\delta=0$ for all subcases.  First, let $r=r_0$ and $s=s_0$.  Equation (\ref{toadex}) is then
\begin{equation*} \label{toad2a}
    x^2+y^2+w^2 = 2\ell + r_0^2 + s_0^2
\end{equation*}
which represents all odd integers, since $r_0$ is odd and $s_0$ is even.

If we then let $r=2r_0$ and $s=2s_0$, Equation (\ref{toadex}) then yields  
\begin{equation*} \label{toad2b}
    x^2+y^2+w^2 = 4\ell + 4(r_0^2 + s_0^2).
\end{equation*}
This allows us to represent all multiples of 4.

If, instead, we let $r=2r_0 - n_2$ and $s=2s_0+n_1$, Equation (\ref{toadex}) then yields  
\begin{equation*} \label{toad2c}
    x^2+y^2+w^2 = 4\ell + (2r_0-n_2)^2 + (2s_0+n_1)^2.
\end{equation*}
As $2r_0-n_2$ and $2s_0+n_1$ are necessarily both odd, this allows us to represent all integers that are $2 \bmod 4$. Combined with the above two choices, this covers all residue classes mod 4, and so similarly to Case (a) we are done.

\underline{Case (c)}: Our modulus here will be 8.  We will need four choices of $r$ and $s$, along with letting $\delta =1$ if $\alpha - z^2 \equiv 3 \bmod 4$.  Note that we are assuming $n_2 \equiv 2 \bmod 4$, so we know that $n_2/2$ is odd.  Additionally, we may assume that $s_0$ is odd and $r_0$ is even.

First, let $r=r_0$ and $s=s_0$.  Equation (\ref{toadex}) is then
\begin{equation} \label{toad3a}
    x^2+y^2+w^2 = 4\ell + r_0^2 + s_0^2.
\end{equation}
If we let $r=r_0 - n_2/2$ and $s = s_0 + n_1/2$, Equation (\ref{toadex}) yields
\begin{equation} \label{toad3b}
    x^2+y^2+w^2 = 4\ell + (r_0-n_2/2)^2 + (s_0+n_1/2)^2.
\end{equation}
Since $s_0$ and $n_2/2$ are both odd, while $r_0$ and $n_1/2$ are even, Equation (\ref{toad3a}) represents all integers that are 1 mod 4, while Equation (\ref{toad3b}) represents all integers that are 2 mod 4.

Next, let $r=2r_0$ and $2s=s_0$.  Equation (\ref{toadex}) is then
\begin{equation*} \label{toad3c}
    x^2+y^2+w^2 = 8\ell + 4(r_0^2 + s_0^2).
\end{equation*}
As $r_0$ is even and $s_0$ is odd, this represents all integers that are 4 mod 8.

If we let $r=2r_0 - n_2$ and $s = 2s_0 + n_1$, Equation (\ref{toadex}) yields
\begin{equation*} \label{toad3d}
    x^2+y^2+w^2 = 8\ell + (2r_0-n_2)^2 + (2s_0+n_1)^2.
\end{equation*}
Since $2r_0 \equiv n_1 \equiv 0 \bmod 4$ and $2s_0 \equiv n_2 \equiv 2 \bmod 4$, this represents all integers that are 0 mod 8, and we therefore have all integers that are 0 mod 4.

We still need to represent integers that are 3 mod 4; this is where $\delta$ comes in.  If we let $\delta = 1$, Equation (\ref{toadex}) becomes
\[x^2 + y^2 + w^2 = 2(r n_1 + s n_2)\ell + r^2 + s^2 - b. \]
If $b \not\equiv 0 \bmod 4$ and $\alpha - z^2 \equiv 3 \bmod 4$, this allows us to represent $\alpha - z^2 + b$ via one of the choices of $r$ and $s$ above.  Therefore we can always represent $\alpha$ as the sum of four squares in $\qab$ in Case (c), which concludes the proof.
\end{proof}

If $a= n_1^2 + n_2^2$ with $\gcd(n_1,n_2) = 2$, then necessarily $a \equiv 0 \bmod 4$; we can then combine Lemmas \ref{43lb} and \ref{toad} to get the following Theorem.

\begin{theorem} \label{frogs}
Suppose that $a = n_1^2 + n_2^2$, where $n_1,n_2 \in \N$ are such that $\gcd(n_1,n_2) = 2$, and 
$m \in \N$.  Then $g_{a,4m+3} = 4$.
\end{theorem}

Specifically, if $n_1 = 0$ and $n_2=2$, we get that $g_{4,4m+3} = 4$ for all $m \in \N$.

\subsection{$\gab(2) = 5$} 

In this section, we find $a,b \in \N$ such that there exists elements of $\qab$ that require 5 squares, which by Theorem \ref{ubsq} gives us that $\gab(2) = 5$.

\begin{theorem} 
For all $m,n \in \N$, there are elements of $Q_{4m,4n}^2$ that are not the sum of four squares in $Q_{4m,4n}$. Therefore $g_{4m,4n} (2) = 5$ for all $m,n \in \N$.
\end{theorem}

\begin{proof}
	Suppose that there exist $w,x,y,z \in Q_{4m,4n}$ such that  $w^2 + x^2 + y^2 + z^2 = 8+2\bk$.  Letting

	\begin{align*}
		&w=w_0+w_1\bi+w_2\bj+w_3\bk\\
		&x=x_0+x_1\bi+x_2\bj+x_3\bk\\
		&y=y_0+y_1\bi+y_2\bj+y_3\bk\\
		&z=z_0+z_1\bi+z_2\bj+z_3\bk,
	\end{align*}

the resulting equations for the real, $\bi$, $\bj$, and $\bk$ coefficients are, respectively:
	\begin{align}
		\begin{split}
			w_{0}^{2}+x_{0}^{2}+y_{0}^{2}+z_{0}^{2}-4m(w_{1}^{2}+x_{1}^{2}+y_{1}^{2}+z_{1}^{2}) \hspace{1in}&\\	-4n(w_{2}^{2}+x_{2}^{2}+y_{2}^{2}+z_{2}^{2})-16mn(w_{3}^{2}+x_{3}^{2}+y_{3}^{2}+z_{3}^{2})&=8	\label{44r}
		\end{split}\\
		w_{0}w_{1}+x_{0}x_{1}+y_{0}y_{1}+z_{0}z_{1}&=0	\label{44i}\\
		w_{0}w_{2}+x_{0}x_{2}+y_{0}y_{1}+z_{0}z_{1}&=0	\label{44j}\\
		w_{0}w_{3}+x_{0}x_{3}+y_{0}y_{3}+z_{0}z_{3}&=1.	\label{44k}
	\end{align}
	
We start by examining Equation (\ref{44k}) mod 2, and note that at least one of $w_0,x_0,y_0,z_0$ must be odd, as otherwise the sum of the terms would be even. Since at least one of these terms must be odd, we assume without loss of generality that $w_{0}\equiv1\bmod2$. With that in mind, Equation (\ref{44r}) mod 8:
	\begin{equation}
		1+x_{0}^{2}+y_{0}^{2}+z_{0}^{2}-4m(w_{1}^{2}+x_{1}^{2}+y_{1}^{2}+z_{1}^{2})-4n(w_{2}^{2}+x_{2}^{2}+y_{2}^{2}+z_{2}^{2})\equiv0\bmod8.	\label{44rm8}
	\end{equation}

	Recall then that for all odd $\ell$, we have $\ell^2\equiv1\bmod8$, and for all even $\ell$, $\ell^2\equiv0\text{ or }4\bmod8$. Since the left side of Equation (\ref{44rm8}) is 1 added to three squares followed by multiples of 4; in order for it to sum to 0 mod 8, $x_{0}^{2},y_{0}^{2},z_{0}^{2}$ must all be $1\bmod8$. So $w_{0}^{2},x_{0}^{2},y_{0}^{2},z_{0}^{2}$ are odd.
    
Then $w_{0}^{2}+x_{0}^{2}+y_{0}^{2}+z_{0}^{2}\equiv4\bmod8$, so an odd number of $w_{1}^{2},x_{1}^{2},y_{1}^{2},z_{1}^{2}$ or $w_{2}^{2},x_{2}^{2},y_{2}^{2},z_{2}^{2}$ must be odd to contribute an additional 4 mod 8. But this forces an odd number of odd terms on the left side of one of Equations (\ref{44i}) and (\ref{44j}), which contradicts their even sums.
	
	Since the equations required for $8+2\bk$ to be a sum of four squares in $Q_{4m,4n}$ cannot hold, $8+2\bk$ cannot be expressed as a sum of four squares in $Q_{4m,4n}$.
 \end{proof}

\section{Other individual cases}

We were able to find $\gab(2)$ in several other cases for specific values of $a$ and $b$.  We include these here for completeness but also to demonstrate the methods used, which vary significantly from those used in Section \ref{valueg}.

\begin{theorem}\label{2223}
$g_{2,2}(2) = g_{2,3}(2) = 3$.
\end{theorem}

These proofs rely of the theory of quadratic forms -- specifically, representations of integers via ternary diagonal quadratic forms.  A ternary diagonal quadratic form is a function $f(x,y,z) = rx^2 + sy^2 + tz^2$; for our purposes, we have $r,s,t \in \N$.  We say a ternary diagonal quadratic form {\em represents} $n \in \N$ if there exists an integer solution to $f(x,y,z) = n$.  Lastly, we say that a ternary diagonal quadratic form is {\em regular} if the only positive integers it does not represent coincide with certain arithmetic progressions. The most common example of this is Legendre's Three-Squares Theorem: that every positive integer not of the form $4^m (8\ell + 7)$ can be represented in the form $x^2 +y^2 + z^2$ with $x,y,z \in \Z$.  For more information on representation of integers via quadratic forms, see \cite{JonesPall} or (more recently) \cite{Hanke}.

Noting that 
\[(x\bi + y\bj + z\bk)^2 = -(ax^2 + by^2 + abz^2),\]
for our Theorem, we will examine the expressions $2x^2 + 2y^2 + 4z^2$ and $2x^2 + 3y^2 + 6z^2$.  Dickson has a complete list of regular diagonal ternary quadratic forms, from whence we get the following Lemma.

\begin{lemma} (Table 5 of \cite{Dickson}) \label{ternreg}
\begin{enumerate}
	\item Let $f_{2,2} (x,y,z) = 2x^2 + 2y^2 + 4z^2$.  Then $f_{2,2}$ represents all even integers not of the form $2 \cdot 4^n (16 \ell + 14)$.
    \item Let $f_{2,3} (x,y,z) = 2x^2 + 3y^2 + 6z^2$.  Then $f_{2,3}$ represents all positive integers not of the form $4^n (8 \ell + 7)$ or $3m+1$.
\end{enumerate}
\end{lemma}

\begin{proof}[Proof of Theorem \ref{2223}]
Let $\al = \alpha_0 + 2\alpha_1 \bi + 2\alpha_2 \bj + 2\alpha_3 \bk \in \qab^2$.  Then, letting $x = 1 + \alpha_1 \bi + \alpha_2 \bj + \alpha_3 \bk$, we have
\begin{equation} \label{2223part}
	\al - x^2 = \alpha_0 - 1 + a\alpha_1^2 + b\alpha_2^2 + ab\alpha_3^2 :=A \in \Z.
\end{equation}
It then suffices to find elements $y,z \in \qab$ with $y = y_0 \in \Z$ and $z= z_1 \bi + z_2 \bj + z_3 \bk$ such that
\begin{equation} \label{2223A}
	A = y^2 + z^2 = y_0^2 - az_1^2 - bz_2^2 - abz_3^2,
\end{equation}
as we would then have $\al = x^2 + y^2 + z^2$.

\underline{Case 1: ($a=b=2$)} In light of Lemma \ref{ternreg} and the regularity of the associated quadratic form, we know that if we can represent the residue class of $A$ mod 32, then we can find $y_0, z_0, z_1, z_2$ that satisfy Equation (\ref{2223A}).

We let $S_{a,b;m}$ be the set of residue classes mod $m$ that are completely represented by $f_{a,b}(z_0,z_1,z_2) = az_0^2 + bz_1^2 + abz_2^2$.  For example, $2 \in S_{2,2;32}$ since $f_{2,2}(1,0,0) = 2$, $2 \not\equiv  2 \cdot 4^n (16 \ell + 14) \bmod 32$ for any $n, \ell \in \N$, and by Lemma \ref{ternreg} $f_{2,2}$ represents all even integers not of the form $2 \cdot 4^n (16 \ell + 14)$.   But $16 \not\in S_{2,2;32}$ since $16 \equiv 2 \cdot 4^1(16 \ell + 14) \bmod 32$.

When $a=b=2$ and $m = 32$, we have
\[S_{2,2;32} = \{2,4,6,8,10,12,14,18,20,22,24,26,30\};\] 
our goal then is to show that for any $A \in \Z$, we can find $y_0 \in \Z$ and $s \in S_{2,2;32}$ such that $A \equiv  y_0^2 - s \bmod 32$.  By Lemma $\ref{ternreg}$, there would then exist $z = z_1\bi+ z_2\bj + z_3\bk \in Q_{2,2}$ such that $-s \equiv z^2 \bmod 32$ and $A = y_0^2 +z^2$.

We can then break this search for $y_0$ and $s$ into cases:
\begin{itemize}
	\item if $A \not\equiv 0,1,4,5,16$, or $17 \bmod 32$, then $A$ is congruent to either $-s$ or $1-s$ for some $s \in S_{2,2;32}$;
	\item if $A \equiv 0,16 \bmod 32$, then $A \equiv 4-s \bmod 32$ for $s=4,20 \in S_{2,2;32}$;
	\item if  $A \equiv 1,5,17 \bmod 32$, then $A \equiv 9-s \bmod 32$ for $s=8,4,24 \in S_{2,2;32}$; and 
	\item if $A \equiv 4 \bmod 32$, then $A \equiv 16-s \bmod 32$ for $s=12 \in S_{2,2;32}$.
\end{itemize}
Therefore we can represent $A$ as a sum of two squares from $Q_{2,2}$, and so we can always express $\alpha$ as a sum of three squares from $Q_{2,2}$.

\underline{Case 2: ($a=2$, $b=3$)}  We again use the set $S_{a,b;m}$, letting $m=24$; this yields
\[S_{2,3;24} = \{2, 3, 5, 6, 9, 11, 14, 17, 18, 21\}.\]  

Similarly to Case 1, we search for $y_0 \in \Z$ and $s \in S_{2,3;24}$ such that $A \equiv  y_0^2 - s \bmod 24$. 
\begin{itemize}
	\item if $A \not\equiv 0,1,2,5,9,12$, or $17 \bmod 24$, then $A$ is congruent to either $-s$ or $1-s$ for some $s \in S_{2,3;24}$;
	\item if $A \equiv 1,2,17 \bmod 24$, then $A \equiv 4-s \bmod 24$ for $s=3,2,11 \in S_{2,3;24}$;
	\item if  $A \equiv 0,12 \bmod 24$, then $A \equiv 9-s \bmod 24$ for $s=9,21 \in S_{2,3;24}$; 
	\item if $A \equiv 5 \bmod 24$, then $A \equiv 16-s \bmod 24$ for $s=11 \in S_{2,3;24}$; and
	\item if $A \equiv 9 \bmod 24$, then $A \equiv 36-s \bmod 24$ for $s=3 \in S_{2,3;24}$.
\end{itemize}
Therefore as above we can always express $\alpha$ as a sum of three squares from $Q_{2,3}$.  Given the lower bound for $\gab(2)$ given by Lemma \ref{lbsq}, we therefore have $\gab(2) = 3$ in both cases.
\end{proof}

The proof of Theorem \ref{2223} relies entirely on the regularity of the associated ternary quadratic forms given in Lemma \ref{ternreg}.  There are, unfortunately, only finitely many regular diagonal ternary quadratic forms (Table 5 of \cite{Dickson} is a complete list), so this exact method has limited general use.  Nonetheless, there does seem to be a close relationship between these Quaternion rings and ternary quadratic forms, and one might be able to relax the regularity condition slightly and be able to represent ``enough'' integers to use a similar method as in Theorem \ref{2223}.

\section{Open Questions}

There are many questions left to explore here.  It seems like it should be possible to find $\gab(2)$ for all $a$ and $b$ positive; at the very least, we'd like to know the proportion of such Quaternion rings that have each of the possible values of $\gab(2)$.    We have also been using as our analog of the integers the Lipschitz Quaternions; the Hurwitz Quaternions would be an equally good choice, especially since we would get unique factorization.  Lastly, we have been focusing on the cases when $\bi^2$ and $\bj^2$ are negative; one could easily investigate the cases when one or both are positive.

\bibliographystyle{plainnat}

\begin{thebibliography}{0}

\bibitem{car}
M. Car, Le probl\`eme de Waring pour l'anneau des polyn\^{o}mes sur un corps fini, in {\it S\'eminaire de Th\'eorie des Nombres, 1972--1973 (Univ. Bordeaux I, Talence), Exp. No. 6}, 13 pp, Lab. Th\'eorie des Nombres, Centre Nat. Recherche Sci., Talence. 

\bibitem{Dickson}
L. E. Dickson, {\it Modern Elementary Theory of Numbers}, Univ. Chicago Press, Chicago, 1939.

\bibitem{Hanke}
J. Hanke, Some recent results about (ternary) quadratic forms, in {\it Number theory}, 147--164, CRM Proc. Lecture Notes, 36, Amer. Math. Soc., Providence, RI, 2004.

\bibitem{eljoseph}
N. Eljoseph, On the representation of a number as a sum of squares, Riveon Lematematika {\bf 7} (1954), 38--43.

\bibitem{JonesPall}
B. W. Jones\ and\ G. Pall, Regular and semi-regular positive ternary quadratic forms, Acta Math. {\bf 70} (1939), no.~1, 165--191. 

\bibitem{nivenquat}
I. Niven, A note on the number theory of quaternions, Duke Math. J. {\bf 13} (1946), 397--400.

\bibitem{nivengauss}
I. Niven, Integers of quadratic fields as sums of squares, Trans. Amer. Math. Soc. {\bf 48} (1940), 405--417.

\bibitem{siegel}
C. Siegel, Darstellung total positiver Zahlen durch Quadrate, Math. Z. {\bf 11} (1921), no.~3-4, 246--275.

\bibitem{sumsofsq}
E. Grosswald, {\it Representations of integers as sums of squares}, Springer, New York, 1985.

\bibitem{wooley}
R. C. Vaughan\ and\ T. D. Wooley, Waring's problem: a survey, in {\it Number theory for the millennium, III (Urbana, IL, 2000)}, 301--340, A K Peters, Natick, MA.

\end{thebibliography}

\end{document}